\documentclass[draft,twoside]{article}
\usepackage{amssymb,amsmath}
\newtheorem{theorem}{Theorem}{}
\newtheorem{lemma}{Lemma}{}
{}
{}
{}
\newtheorem{example}{Example}{}
\def\rmd{\mathrm{d}}
\def\a{\alpha}
\def\b{\beta}
\def\D{\Delta}

\def\f{\varphi}

\def\t{\tau}
\def\z{\zeta}

\def\dfrac{\displaystyle\frac}
\def\dsum{\displaystyle\sum}
\def\dint{\displaystyle\int}

\def\G{\Gamma}
\newcommand{\Res}{\mathop{\rm Res}}

\def\lt{{\cal L}}
\def\gt{{\cal G}}
\def\mt{{\cal M}}
\def\rt{{\cal R}}

\newcommand\arctg{\mathop{\rm arctg}}
\newcommand\res{\mathop{\rm Res}}
\newcommand\re{\mathop{\rm Re}}

\title{On the Laplace-type transform and its applications}%{On the gamma transform and its applications}

\author{Slobodan B. Tri\v ckovi\'c$^1$, Miomir S. Stankovi\'c$^2$\\[2mm]
$^1$\small\sl University of Ni\v s, Department of Mathematics, Serbia, e-mail: sbt@junis.ni.ac.rs\\
$^2$\small\sl Mathematical Institute of the Serbian Academy of Sciences and Arts, Belgrade, Serbia}

\date{\today}

\begin{document}

\maketitle

\begin{abstract} We use the Laplace transform and the Gamma function to introduce a new integral transform and 
name it the Laplace-type transform possessing the property of mapping a function to a functional sequence, 
which cannot be achieved by the Laplace transform. In addition, we construct a backward difference as 
a generalization of the backward difference operator $\nabla$, and connecting it to the Laplace-type transform, 
we deduce a method for solving difference equations and, relying on classical orthogonal polynomials, for obtaining 
combinatorial identities. A table of some elementary functions and their images is in the Appendix.
\smallskip

\noindent KEY WORDS: Gamma function, Laplace transform, backward difference operator.
\end{abstract}

\section{Introduction and preliminaries}

Very often, if we interpret continuity problems as discrete, we can solve them more easily. For example,
complicated definite integrals of continuous functions are calculated numerically by being reduced to finite sums.
Also, the opposite is true, i.e. the only way to solve certain discrete problems is to view them in the light of
continuity. For this purpose, we develop a specific transform and its inverse, enabling us to map the continuous
functions into sequences and vice versa.

To start with, it is known that the Gamma function is defined by the integral
\begin{equation}\label{gamma-f}
\G(\a)=\int_0^\infty e^{-x}x^{\a-1}\,\rmd x\quad(\a>0).
\end{equation}
We can regard it as the Mellin transform of the function $e^{-x}$, i.e. $\G(\a)=\mt\{e^{-x}\}(\a)$.
On the other hand, for $p>0$, after introducing the substitution $x=p\,t$, we obtain
$$
\G(\a)=\int_0^\infty s^\a e^{-s\,t}t^{\a-1}\,\rmd t\quad\Rightarrow\quad
\frac1{s^\a}=\int_0^\infty e^{-s\,t}\frac{t^{\a-1}}{\G(\a)}\,\rmd t=\lt\Big\{\frac{t^{\a-1}}{\G(\a)}\Big\}(s),
$$
and thereby we have determined the Laplace transform of the function $t^{\a-1}/\G(\a)$.

However, if apart from $t^{\a-1}$ there appears a real-valued piecewise continuous function $f(t)$
over $(0,\infty)$, and of exponential order $r$, i.e. there exist positive numbers $M$ and $r$ such
that the inequality $|f(t)|<Me^{rt}$ holds, according to and in keeping with \cite[p.~744, Eq (2.1)]{Temme},
we introduce a new integral transform $\gt$ of the function $f(t)$ and name it the \textbf{Laplace-type transform}
\begin{equation}\label{gamma-t}
\gt\{f(t)\}(s)=\frac{1}{\G(\a)}\int_0^\infty e^{-s\,t}\,t^{\a-1}\,f(t)\,\rmd t\quad(s>r,\,\a>0).
\end{equation}
Under the given conditions, this integral is convergent, since
$$
\int_0^\infty e^{-s\,t}\,\frac{t^{\a-1}}{\G(\a)}\,f(t)\,\rmd t<
M\int_0^\infty e^{-(s-r)\,t}\,\frac{t^{\a-1}}{\G(\a)}\,\rmd t=\frac1{(s-r)^\a}.
$$

By setting $\a-1=n$ in \eqref{gamma-t} for $n\in\mathbb N\cup\{0\}$, the Laplace-type transform maps
a function $f(t)$ to a sequence of functions $\{\f_n(s)\}_{n\in\mathbb N_0}$, given by
\begin{equation}\label{1}
\f_n(s)=\int_0^\infty e^{-s\,t}\,\frac{t^n}{n!}\,f(t)\,\rmd t\quad(s>0),
\end{equation}
i.e. $\f_n(s)\fallingdotseq f(t)$.
For $s=1$ we write  $\f_n$ instead of $\f_n(1)$ so that in this case, the Laplace-type transform maps
the function $f(t)$ to a sequence of numbers $\{\f_n\}_{n\in\mathbb N_0}$, i.e. $\f_n\fallingdotseq f(t)$,
and later within the section Applications it will enable us to solve difference equations through
differential ones.

\begin{example} The Hurwitz zeta function\index{function!Hurwitz zeta} $\z(s,a)$ initially defined
for $\sigma>1$ ($s=\sigma+i\tau$) by the series having its  integral representation
\begin{equation}\label{hurwitz-sum}
\z(s,a)=\sum_{k=0}^\infty\frac1{(k+a)^s},
\end{equation}
where $a$ is a fixed real number, $0<a\leqslant 1$.

Multiplying the equality \eqref{hurwitz-sum} by $\G(s)$,
\begin{equation*}
\z(s,a)\G(s)=\sum_{k=0}^{+\infty}\frac1{(k+a)^s}\int_0^\infty x^{s-1}e^{-x}\,\rmd x,
\end{equation*}
and introducing the substitution $x=(k+a)t$, then interchanging the sum and integral, dividing by $\G(z)$, 
we get the integral representation of the Hurwitz zeta function \begin{equation}\label{hurwitz-int}
\z(s,a)=\int_0^\infty t^{s-1}\sum_{k=0}^\infty e^{-(k+a)t}\,\rmd t
=\frac1{\G(s)}\int_0^\infty\frac{t^{s-1}e^{-at}}{1-e^{-t}}\,\rmd t.
\end{equation}
This integral can be viewed as a Mellin transform of the function $\dfrac{e^{-ax}}{1-e^{-x}}$.

Referring to \eqref{gamma-t} for $f(t)=\dfrac1{1-e^{-t}}$, $\a=s$, $s=a$, and taking into account \eqref{hurwitz-int}, we have
\begin{equation}\label{5}
\gt\{f(t)\}(s)=\gt\left\{\frac1{1-e^{-t}}\right\}(\a)=\z(\a,a).
\end{equation}
However, taking $f(t)=\dfrac1{e^t-1}$, we find
$$
\gt\left\{\frac1{e^t-1}\right\}(s)=\frac1{\G(s)}\int_0^\infty\frac{e^{-at}}{e^t-1}t^{s-1}\,\rmd t
=\frac1{\G(s)}\int_0^\infty\frac{e^{-(a+1)t}}{1-e^{-t}}t^{s-1}\,\rmd t=\z(s,a+1).
$$

On the other hand, the generating function for the Bernoulli polynomials is
$$
\frac{te^{at}}{e^t-1}=\sum_{k=0}^\infty B_k(a)\frac{t^k}{k!},
$$
which implies
\begin{equation}\label{6}
\z(s,a+1)=\frac1{\G(s)}\int_0^\infty\frac{te^{a\,t}}{e^t-1}e^{-2a\,t}t^{s-2}\,\rmd t
=\frac1{\G(s)}\int_0^\infty e^{-2a\,t}t^{s-2}\sum_{k=0}^\infty B_k(a)\frac{t^k}{k!}\,\rmd t.
\end{equation}
Interchanging summation and integration, we rewrite \eqref{6} in the form of
$$
\z(s,a+1)=\sum_{k=0}^\infty\frac{B_k(a)}{\G(s)}\int_0^\infty e^{-2a\,t}\frac{t^{k}}{k!}t^{s-2}\,\rmd t
=\sum_{k=0}^\infty\frac{B_k(a)}{2^{s+k-1}\G(s)}\int_0^\infty e^{-a\,t}\frac{t^{k}}{k!}t^{s-2}\,\rmd t.
%=\gt\{x^{s-2}\}(s).
$$
After introducing substitution $x=at$ in the last integral, in view of \eqref{1} and Appendix,
%$$
%\int_0^\infty e^{-2a\,t}t^{k+s-2}\,\rmd t
%=\frac1{2^s}\gt\Big\{\Big(\frac x2\Big)^{k-1}\Big\}(s)=\frac{\G(s+k-1)}{(2a)^{s+k-1}}.
%$$
we have
$$
\z(s,a+1)=\sum_{k=0}^\infty\frac{B_k(a)}{\G(s)}\frac{\G(s+k-1)}{(2a)^{s+k-1}k!}
=\frac{\G(s-1)}{(2a)^{s-1}\G(s)}+\sum_{k=1}^\infty\frac{B_k(a)}{\G(s)}\frac{\G(s+k-1)}{(2a)^{s+k-1}k!}.
$$
Using the relation $B_k(a)=-k\z(1-k,a)$, $k\in\mathbb N$, one obtains the representation
$$
\z(s,a+1)=\frac1{(2a)^{s-1}(s-1)}-\sum_{k=1}^\infty\frac{\z(1-k,a)}{(2a)^{s+k-1}}\frac{(s)_{k-1}}{k!},
\quad (s)_{k-1}=\frac{\G(s+k-1)}{\G(s)},
$$
with $(x)_n$ being the Pochhammer symbol.
\end{example}

\section{Basic properties}

We shall show later how the Laplace-type transform helps us derive a method for solving difference equations
and obtaining combinatorial identities. For this purpose, we need to investigate some of its properties.

\begin{lemma}\label{lm1} Let $f(t)$ be piecewise continuous on $[0,\infty)$ of exponential order $r$
and $F(s)$ denote its Laplace transform, i.e. $F(s)=\lt\{f(t)\}(s)$. For its $n$th derivative $F^{(n)}(s)$
there holds % F^{(n)}(s)=For $n\in\mathbb N$,
\begin{equation}\label{3}
\f_n(s)=\frac{(-1)^n}{n!}F^{(n)}(s), %,\quad\gt\{f\}=\frac{(-1)^n}{n!}\lt\{f(t)\}^{(n)}.
\end{equation}
where $\f_n(s)\fallingdotseq f(t)$.
\end{lemma}

\begin{proof} By differentiating $n$ times $F(s)$, we have
$$
F^{(n)}(s)=\frac{d^n}{ds^n}F(s)=\frac{d^n}{ds^n}\left(\int_0^\infty e^{-s\,t}f(t)\,dt\right)
=(-1)^n\int_0^\infty e^{-s\,t}t^nf(t)\,dt. %=(-1)^nf_n,
$$
By the hypotheses, for $s\geq x_0>r$, it is justified to interchange
the derivative and integral sign in the last calculation (see \cite{Shiff}).
In view of \eqref{1}, %and \eqref{s-1},
there immediately follows \eqref{3}.
\end{proof}

\begin{lemma}\label{lm2}Let $g(t)=\dint_0^t f(x)\,dx$. If $f(t)\fallingdotseq\f_n(s)$ and $g(t)\fallingdotseq\psi_n(s)$, then
$$
\psi_n(s)=\sum_{k=0}^n\frac{\f_k(s)}{s^{n-k+1}}.
$$
\end{lemma}

\begin{proof}Let $F_f(s)$ and $F_g(s)$ denote the Laplace transform of $f(t)$ and $g(t)$
respectively. Considering that
$$
F_g(s)=\frac1sF_f(s),
$$
and making use of \eqref{1}, \eqref{3} and the Leibnitz rule, we find $\psi_n(s)\fallingdotseq g(t)$,
where
$$
\psi_n(s)=\frac{(-1)^n}{n!}\frac1sF_f^{(n)}(s)=\sum_{k=0}^n\frac1{s^{n-k+1}}\frac{(-1)^k}{k!}F_f^{(k)}(s)
=\sum_{k=0}^n\frac{\f_k(s)}{s^{n-k+1}},
$$
which proves the statement. \end{proof}

The convolution of two functions, $f(t)$ and $g(t)$, defined for $t>0$, plays an important role
in numerous different physical applications. The convolution is given by the integral
$$
(f*g)(t)=\int_0^t f(t-\t)g(\t)\,d\t,
$$
which exists if $f$ and $g$ are piecewise continuous.

\begin{lemma}\label{lm3}Let $\f_n(s)\fallingdotseq f(t)$ and $\psi_n(s)\fallingdotseq g(t)$.
The following equality holds
$$
\sum_{k=0}^n\f_{n-k}(s)\psi_k(s)\fallingdotseq(f*g)(t).
$$
\end{lemma}

\begin{proof}Let $F_f(s)$, $F_g(s)$ and $F_{f*g}(s)$ denote the Laplace transform of $f(t)$, $g(t)$
and $(f*g)(t)$ respectively. It is well known that
$$
F_{f*g}(s)=F_f(s)\cdot F_g(s).
$$
Using the result of Lemma \ref{lm1}, we obtain
$$
h_n(s)=\frac{(-1)^n}{n!}\big(F_f(s)\cdot F_g(s)\big)^{(n)}\fallingdotseq (f*g)(t),
$$
and by using the Leibnitz rule and \eqref{3}, we have
$$
h_n(s)=\sum_{k=0}^n\frac{(-1)^{n-k}F_f^{(n-k)}(s)}{(n-k)!}\frac{(-1)^kF_g^{(n-k)}(s)}{k!}
=\sum_{k=0}^n\f_{n-k}(s)\psi_k(s),
$$
which ends the proof. \end{proof}

\section{The Laplace-type transform and Laguerre\\ polynomials}

We recall that the Laguerre polynomials, the explicit representation of which is
$$
L_n(x)=\sum_{k=0}^n(-1)^k\binom nk\frac{x^k}{k!},
$$
by introducing an inner product as follows
$$
(L_m,L_n)=\int_0^\infty e^{-x}L_m(x)L_n(x)\,dx,
$$
we make them into an orthogonal system with respect to the weight function $e^{-x}$, whereby
they form an orthogonal base of the Hilbert space $L_2(0,\infty;e^{-x})$. To connect them
to the Laplace-type transform, we replace $x$ with $s\,x$, and the above inner product becomes
$$
\int_0^\infty e^{-s\,x}\,s\,L_m(s\,x)L_n(s\,x)\,dx.
$$
By this, we construct a system of Laguerre functions over $(0,\infty)$, i.e.
\begin{equation}\label{laguerre*}
L_n^*(x,s)=(-1)^n\sqrt s\,L_n(s\,x)\quad(s>0),
\end{equation}
and define their inner product by
$$
(L_m^*,L_n^*)=\int_0^\infty e^{-s\,x}(-1)^m\sqrt s\,L_m(s\,x)(-1)^n\sqrt s\,L_n(s\,x)\,dx.
$$
Since we easily find $(L_m^*,L_n^*)=0,\,m\ne n$, $(L_n^*,L_n^*)=1$, we conclude the system $L_n^*(x)$
($n\in\mathbb N_0$) form an orthonormal basis of the Hilbert space with respect to the weight function
$e^{-s\,x}$. A function $f(x)$ of the exponential order $r$ belongs to $L_2(0,\infty;e^{-s\,x})$,
because if there exist positive constants $M$ and $r$ such that $|f(x)|<Me^{r\,x}$, then
$$
\|f\|^2=\int_0^{+\infty}e^{-s\,x}|f(x)|^2\,dx<M^2\int_0^{+\infty}e^{-(s-2r)x}\,dx=\frac{M^2}{s-2r}\quad(s>2r),
$$
and can be expanded into the Fourier series over the functions $L_{n,s}^*(x)$, so that we have
\begin{equation}\label{expansion}
f(x)=\sum_{k=0}^\infty a_k(s)L_{k,s}^*(x),\quad a_n(s)
=(-1)^n\int_0^\infty e^{-s\,x}f(x)\sqrt s\,L_n(s\,x)\,dx.
\end{equation}%Making use of \eqref{laguerre*}, there follows %\begin{equation}\label{an}\end{equation}
Also, the Bessel inequality holds
\begin{equation}\label{bess-ineq}
\sum_{k=0}^\infty a_k^2(s)\leq\|f\|^2.
\end{equation}

To prove the expansion \eqref{expansion}, we form its $n$th partial sum, and after rearrangement obtain
$$
S_n(x;f)=\int_0^\infty e^{-s\,t}f(t)\left[\sum_{k=0}^n L_k^*(t,s)L_k^*(x,s)\right]\,dt
$$
and applying the Christoffel-Darboux formula
$$
\sum_{k=0}^n L_k^*(t,s)L_k^*(x,s)=(n+1)\frac{L_{n+1}^*(x,s)L_n^*(t,s)-L_{n+1}^*(t,s)L_n^*(x,s)}{t-x}
=(n+1)\frac{\Phi_n(x,t)}{t-x},
$$
we have
$$
S_n(x;f)=(n+1)\int_0^\infty e^{-s\,t}f(t)\frac{\Phi_n(x,t)}{t-x}\,dt.
$$
Further, we consider the difference
$$
S_n(x;f)-f(x)=(n+1)\int_0^\infty e^{-s\,t}\,\frac{f(x)-f(t)}{x-t}\,\Phi_n(x,t)\,dt.
$$

\begin{theorem}The sequence $\{a_n(s)\}$ can be expressed in terms of the sequence $\{\f_n(s)\}$
defined by \eqref{1}, and the other way round
\begin{equation}\label{a-fi}
a_n(s)=\sum_{k=0}^n(-1)^{n-k}\binom nk s^{k+\frac12}\f_k(s),\qquad
\f_n(s)=\frac{\sqrt s}{s^{n+1}}\sum_{k=0}^n\binom nk a_k(s).
\end{equation}
\end{theorem}

\begin{proof}We prove the first equality of \eqref{a-fi}. Relying on \eqref{expansion} and then
referring to \eqref{1}, we obtain
{\arraycolsep0.11em
\begin{eqnarray*}
a_n(s)&=&(-1)^n\sqrt s\int_0^{+\infty}e^{-s\,x}f(x)\sum_{k=0}^n(-1)^k\binom nk\frac{(s\,x)^k}{k!}\,dx\\
&=&\sum_{k=0}^n(-1)^{n-k}\binom nk s^{k+\frac12}\int_0^{+\infty}e^{-s\,x}\,\frac{x^k}{k!}\,f(x)\,dx
=\sum_{k=0}^n(-1)^{n-k}\binom nk s^{k+\frac12}\f_k(s).
\end{eqnarray*}}

To derive the second equality, we consider the first equality in \eqref{a-fi} as the binomial transform
of the sequence $\{b_n(s)\}$, i.e.
$$
a_n(s)=\sum_{k=0}^n(-1)^{n-k}\binom nk b_k(s),\quad b_k(s)=s^{k+\frac12}\f_k(s).
$$
So its inverse transform is
$$
b_n(s)=\sum_{k=0}^n\binom nk a_k(s)\hskip0.17cm\Rightarrow\hskip0.17cm s^{n+\frac12}\f_n(s)
=\sum_{k=0}^n\binom nk a_k(s)\hskip0.17cm\Leftrightarrow\hskip0.17cm
\f_n(s)=\frac{\sqrt s}{s^{n+1}}\sum_{k=0}^n\binom nk a_k(s).
$$
whereby we complete the proof. \end{proof}

\section{The inverse Laplace-type transform}

Let $S_c$ be the space of sequences of functions $\{t_n(s)\}_{n\in\mathbb N_0}$ such that
\begin{equation}\label{s-c}
\sum_{k=0}^{+\infty}\Big\{\D_n^k\Big(s^n t_n(s)\Big)_{n=0}\Big\}^2<+\infty,
\end{equation}
where $\D_n$ is the forward difference operator
$$
\D_n t_n(s)=t_{n+1}(s)-t_n(s),\quad\D_n^kt_n(s)=\D_n(\D_n^{k-1}t_n(s))
=\sum_{j=0}^k(-1)^{k-j}\binom kj t_{n+j}(s).
$$
We shall show that the sequence $\{\f_n(s)\}_{n\in\mathbb N_0}$ defined by \eqref{1} satisfies \eqref{s-c},
i.e. belongs to the space $S_c$.

Taking account of \eqref{1} and \eqref{expansion}, we find
{\arraycolsep0.11em
\begin{eqnarray*}
\D_n^k\Big(s^n\f_n(s)\Big)_{n=0}&=&\sum_{j=0}^k(-1)^{k-j}\binom kj s^j\f_j(s)
=(-1)^k\sum_{j=0}^k\frac{(-1)^j}{j!}\binom kj s^j\int_0^\infty e^{-s\,x}x^jf(x)\,dx,\\
&=&(-1)^k\int_0^\infty e^{-s\,x}\,f(x)\sum_{j=0}^k\frac{(-1)^j}{j!}\binom kj(s\,x)^j\,dx,\\
&=&\frac{(-1)^k}{\sqrt s}\int_0^\infty e^{-s\,x}\,f(x)\sqrt s\,L_k(s\,x)\,dx=\frac{a_k(s)}{\sqrt s}.
\end{eqnarray*}
}so on this basis and the Bessel inequality \eqref{bess-ineq}, we finally have
$$
\sum_{k=0}^\infty\Big\{\D_n^k\Big(s^n\f_n(s)\Big)_{n=0}\Big\}^2
=\frac1s\sum_{k=0}^\infty a_k^2(s)<+\infty.
$$

We make use of the relations \eqref{a-fi} between orthogonal Laguerre functions \eqref{laguerre*} and
the Laplace-type transform to define a mapping $\rt: S_c\mapsto L_2(0,+\infty)$ as follows
\begin{equation}\label{7}
f(x)=\rt\{\f_n\}(x)=\sum_{k=0}^\infty\D_n^k\Big(s^n\f_n\Big(\frac s2\Big)\Big)_{n=0}\sqrt s\,L_{k,s}^*(x)%\\
\end{equation}

\begin{theorem}\label{thm4}The transforms $\gt$ and $\rt$ are inverse to each other, in other words,
if $\f_n(s)$ is obtained by $\gt\{f(t)\}(s)$, then $f(t)$ is obtained by $\rt\{\f_n(s)\}(t)$,
and the other way round, if \eqref{7} holds true, then $f(t)\risingdotseq\f_n(s)$.
\end{theorem}

\begin{proof}Let $f(t)\risingdotseq\f_n(s)$. Then the first relation in \eqref{a-fi} yields
$$
a_k(s)%=(-1)^k\sqrt s\sum_{j=0}^k(-s)^j\binom kj\f_j(s)
=\sqrt s\sum_{j=0}^k(-1)^{k-j}\binom kj s^j\f_j(s)=\sqrt s\,\D_n^k(s^n\f_n(s))\mid_{n=0}.
$$
Considering that $a_k(s)$ are coefficients in the expansion of $f(x)$ in terms of Laguerre polynomials
$$
f(x)=\sum_{k=0}^\infty a_k(s)L_{k,p}^*(x)=\sqrt s
\sum_{k=0}^\infty\D_n^k(s^n\f_n(s))\mid_{n=0}L_{k,s}^*(x)=\rt\{\f_n\}(x),
$$
we prove the first direction.

Conversely, let $\rt\{\f_n\}(x)=f(x)$. From \eqref{7}, and in view of \eqref{laguerre*}, there follows
{\arraycolsep0.11em
\begin{eqnarray*}
\f_n(s)&=&\int_0^\infty e^{-px}x^nf(x)\,dx
=\sqrt s\sum_{k=0}^\infty\D_n^k(s^n\f_n(s))_{n=0}\int_0^\infty e^{-s\,x}x^nL_{k,s}^*(x)\,dx\\
&=&s\sum_{k=0}^\infty (-1)^k\D_n^k(s^n\f_n(s))_{n=0}\int_0^\infty x^n e^{-px}L_k(sx)\,dx\\
&=&\frac1{s^n}\sum_{k=0}^\infty (-1)^k\D_n^k(s^n\f_n(s))_{n=0}\int_0^\infty (st)^n e^{-px}L_k(sx)\,d(sx).
\end{eqnarray*}
}However, introducing the substitution $t=s\,x$ in the last integral, it becomes
$$
\int_0^\infty e^{-t}t^nL_k(t)\,dt=\int_0^\infty e^{-st}t^n\Big(\sum_{j=0}^k(-1)^j\binom kj\frac{t^j}{j!}\Big)\,dt
=\sum_{j=0}^k\frac{(-1)^j}{j!}\binom kj\int_0^\infty e^{-t}t^{n+j}\,dt,
$$
but the right-hand side integral is actually $\G(n+j+1)=(n+j)!$, and we have
{\arraycolsep0.11em
\begin{eqnarray*}
\int_0^\infty e^{-t}t^nL_k(t)\,dt&=&\sum_{j=0}^k\frac{(-1)^j}{j!}\binom kj(n+j)!
=n!\sum_{j=0}^k(-1)^j\binom kj\binom{n+j}j\\
&=&(-1)^kn!\sum_{j=0}^k(-1)^{k-j}\binom kj\binom{n+j}j=(-1)^k n!\binom nk,\quad k\leqslant n,
\end{eqnarray*}
}so that one obtains
$$
\gt\{f(t)\}(s)=\frac1{n!}\int_0^\infty e^{-s\,x}x^nf(x)\,dx=\frac1{s^n}\sum_{k=0}^n\binom nk\D^k(s^n\f_n(s))_{n=0}
=\frac1{s^n}s^n\f_n(s)=\f_n(s).
$$
Hereby we complete the proof. \end{proof}

We make use of the formula for the inverse Laplace transform (see \cite{Mitr})
\begin{equation}\label{resudue}
f(t)=\frac1{2\pi i}\int_{c-i\infty}^{c+i\infty}e^{s\,t}F(s)ds=\sum_{j=0}^k\res_{s=s_j}e^{s\,t}F(s),
\end{equation}
to prove

\begin{lemma}The inverse Laplace-type transform $\rt$ is determined by the equality
$$
f(t)=\frac{n!}{t^n}\sum_{j=0}^k\res_{s=s_j}e^{s\,t}\f_n(s).
$$
\end{lemma}

\begin{proof}By virtue of $n!f(t)\risingdotseq n!\f_n(s)$, and relying on {\rm Lemma \ref{lm1}},
we have already seen that $n!\f_n(s)=\lt\{t^n f(t)\}(s)$, whence we have $\lt^{-1}\{n!\f_n(s)\}(t)=t^nf(t)$.
Referring to \eqref{resudue}, we find
$$
t^nf(t)=\frac1{2\pi i}\int_{c-i\infty}^{c+i\infty}e^{st}{n!}\f_n(s)ds
={n!}\sum_{j=0}^k\res_{s=s_j}e^{st}\f_n(s).
$$
However, taking into account {\rm Theorem \ref{thm4}}, we conclude that $f(t)=\rt\{\f_n(s)\}(t)$, i.e.
$$
\rt\{\f_n(s)\}(t)=\frac{n!}{t^n}\sum_{j=0}^k\res_{s=s_j}e^{st}\f_n(s),
$$
whereby we prove the lemma. \end{proof}

\section{The Laplace-type transform and a backward\\ difference operator}

For a sequence $\{\f_n(s)\}$, the backward shift operator $E^{-1}$ is defined by $\f_{n-1}(s)=E^{-1}\f_n(s)$.
We define the generalized backward difference operator $\nabla_s$ as
\begin{equation}\label{nabla-s}
\nabla_s\f_n(s)=s\f_n(s)-\f_{n-1}(s),
\end{equation}
which for $s=1$ reduces to the backward difference operator $\nabla\f_n=\f_n-\f_{n-1}$, 
with $\f_n=\f_n(1)$.

By writing \eqref{nabla-s} in the form of
$$
\nabla_s\f_n(s)=s\f_n(s)-E^{-1}\f_n(s)=(sI-E^{-1})\f_n(s),
$$
with $I$ as identity operator, i.e. $I_s\f_n(s)=\f_n(s)$, for the power $\nabla_s^p\f_n(s)$
($p\in\mathbb N$), we have
{\arraycolsep0.11em
\begin{eqnarray*}
\nabla_s^p\f_n(s)&=&(s I-E^{-1})^p\f_n(s)=\sum_{k=0}^p(-1)^k\binom pk s^{p-k}(E^{-1})^k\f_n(s)\\
&=&\sum_{k=0}^p(-1)^k\binom pk s^{p-k}E^{-k}\f_n(s)=\sum_{k=0}^p(-1)^k\binom pk s^{p-k}\f_{n-k}(s).
\end{eqnarray*}}On this basis, starting with $\f_{n-1}(s)=s\f_n(s)-\nabla_s\f_n(s)$, and using
the method of mathematical induction, we can prove that $\f_{n-p}(s)$ is expressed in terms of
the powers $\nabla_s\f_n(s)$
\begin{equation}\label{fn-p-nabla-p-fn}
\f_{n-p}(s)=s^p\f_n(s)-\binom p1s^{p-1}\nabla_s\f_n(s)+\binom p2s^{p-2}\nabla_s^2\f_n(s)+\cdots+(-1)^p\nabla_s^p\f_n(s),
%=\sum_{j=0}^p(-1)^j\binom pjs^{p-j}\nabla_s^j f_n
\end{equation}
and we notice that $\nabla_s^p\f_n(s)$ for $s=1$ becomes $\nabla^p f_n$.

\begin{theorem}\label{th2}Let $f(t)\risingdotseq\f_n(s)$. Then there holds
$$
f^{(s)}(t)\risingdotseq
\left\{\begin{array}{l}
\nabla_s^p\f_n(s),\quad n\geqslant p\\
\nabla_s^p\f_n-(-1)^n\dsum_{j=1}^{p-n}\binom{p-j}n s^{p-j}f^{(j-1)}(0),\quad n\leqslant p-1.
\end{array}\right.
$$
%In other words, $p${\rm th} derivative of a function $f$ is mapped by the Laplace-type transform $\gt$
%to $p${\rm th} backward difference of the $\f_n$ defined by \eqref{1}, if $p\leqslant n$.
\end{theorem}

\begin{proof}We use the Laplace transform of the $p$th derivative of the function $f$, that is
\begin{equation}\label{laplace-p}
\lt\{f^{(s)}(t)\}(s)=s^p\lt\{f(t)\}(s)-\sum_{j=1}^p s^{p-j}f^{(j-1)}(0).
\end{equation}
By virtue of this and \eqref{3}, there follows
{\arraycolsep0.11em
\begin{eqnarray*}
f^{(s)}(t)&\risingdotseq&\frac{(-1)^n}{n!}\lt\{f^{(s)}(t)\}^{(n)}(s)\\
&=&\frac{(-1)^n}{n!}\frac{d^n}{ds^n}\Big(s^p\lt\{f(t)\}(s)\Big)
-\frac{(-1)^n}{n!}\frac{d^n}{ds^n}\sum_{j=1}^p s^{p-j}f^{(j-1)}(0).
\end{eqnarray*}}However, for $n\geqslant p$ there holds
$$
\frac{d^n}{ds^n}\sum_{j=1}^p s^{p-j}f^{(j-1)}(0)=0,
$$
so in this case, we find
{\arraycolsep0.11em
\begin{eqnarray*}
f^{(s)}(t)&\risingdotseq&\frac{(-1)^n}{n!}\big(s^p\lt\{f(t)\}(s)\big)^{(n)}
=\frac{(-1)^n}{n!}\sum_{j=0}^n\binom nj(s^p)^{(j)}\lt\{f(t)\}^{(n-j)}(s)\\
&=&\sum_{j=0}^n(-1)^j\frac{p^{(j)}}{j!}s^{p-j}\frac{(-1)^{n-j}}{(n-j)!}\lt\{f(t)\}^{(n-j)}(s)
=\sum_{j=0}^p(-1)^j\binom pj s^{p-j}\f_{n-j}(s)\\
&=&\nabla_s^p\f_n(s),
\end{eqnarray*}
}where $p^{(j)}$ stands for the falling factorial, i.e. $p^{(j)}=p(p-1)\cdots(p-j+1)$. Taking $s=1$,
we obtain $f^{(s)}(t)\risingdotseq\nabla^p f_n$.

Otherwise, for $n\leqslant p-1$, we have
{\arraycolsep0.11em
\begin{eqnarray*}
\gt\{f^{(s)}(t)\}(s)&=&\frac{(-1)^n}{n!}\frac{d^n}{ds^n}\Big(s^p\lt\{f(t)\}(s)\Big)
-\frac{(-1)^n}{n!}\frac{d^n}{ds^n}\sum_{j=1}^p s^{p-j}f^{(j-1)}(0)\\
&=&\sum_{j=0}^n(-1)^j\frac{p^{(j)}}{j!}s^{p-j}\frac{(-1)^{n-j}}{(n-j)!}\lt\{f(t)\}^{(n-j)}(s)\\
&&\hskip3cm-(-1)^n\sum_{j=1}^{p-n}\frac{(p-j)^{(n)}}{n!}s^{p-j-n}f^{(j-1)}(0).%\\
\end{eqnarray*}}The Gamma function \eqref{gamma-f} often fills in the values of the factorial, and can be
analytically extended to the whole complex plane, except for the non-positive integers, where it has poles.
However, its reciprocal function has there a value of zero. So, this allows the summation index $j$ in the first sum to run up to $p$ because the values of $1/(n-j)!=1/\Gamma(n-j+1)$ for $j=n+1,\dots,p$ are equal to zero! In that case, this sum is actually $\nabla_s^p\f_n(s)$, so that we have
$$
f^{(s)}(t)\risingdotseq\nabla_s^p\f_n(s)-(-1)^n\sum_{j=1}^{p-n}\binom{p-j}n s^{p-j-n}f^{(j-1)}(0),
$$
%&=&\nabla_s^k\f_n(s)-\frac{(-1)^n}{s^n}\sum_{j=1}^k\binom{k-j}ns^{k-j}f^{(j-1)}(0),
whereby we complete the proof. \end{proof}

\begin{lemma}If $f(t)\risingdotseq\f_n(s)$ and $r\in\mathbb N$, then $t^rf(t)\risingdotseq\nabla_s^p\f_{n+r}(s)$
for $r+n\geqslant p$. If $r+n\leqslant p-1$ there holds %\label{lm4}
$$
t^rf^{(s)}(t)\risingdotseq(n+1)_r\Big(\nabla_s^p\f_{n+r}(s)-(-1)^{n+r}
\sum_{j=1}^{p-n-r}\binom{p-j}{n+r}s^{p-j-n-r}f^{j-1}(0)\Big).
$$
where $(n+1)_r=(n+1)(n+2)(n+3)\cdots(n+r),\,r\in\mathbb N$. In the special case, for $p=0$, we have
$$
t^rf(t)\risingdotseq(n+1)_r\f_{n+r}(s).
$$
\end{lemma}

\begin{proof} Making use of one of the basic properties of
the Laplace transform %that we dealt with in the proof of Lemma \ref{lm1}
$$
\lt\{t^r g(t)\}(s)=(-1)^r\frac{d^r}{ds^r}\lt\{g(t)\}(s),
$$
whence, after replacing $g(t)$ with $f^{(s)}(t)$ and applying \eqref{3}, we obtain
{\arraycolsep0.11em
\begin{eqnarray*}
t^rf^{(s)}(t)&\risingdotseq&\frac{(-1)^n}{n!}\frac{d^n}{ds^n}\lt\{t^r f^{(s)}(t)\}(s)
=\frac{(-1)^n}{n!}\frac{d^n}{ds^n}(-1)^r\frac{d^r}{ds^r}\lt\{f^{(s)}(t)\}\\
&=&\frac{(-1)^{n+r}}{n!}\frac{d^{r+n}}{ds^{r+n}}\lt\{f^{(s)}(t)\}(s)
=(n+1)_r\frac{(-1)^{n+r}}{(n+r)!}\frac{d^{r+n}}{ds^{r+n}}\lt\{f^{(s)}(t)\}(s).
\end{eqnarray*}}For $p=0$, on the basis of Lemma \ref{lm1}, one obtains  the equality
$t^rf(t)\risingdotseq(n+1)_r\f_{n+r}(s)$. Let $p\geqslant 1$. In view of \eqref{laplace-p},
we find
$$
t^rf^{(s)}(t)\risingdotseq
(n+1)_r\frac{(-1)^{n+r}}{(n+r)!}\left(\frac{d^{r+n}}{ds^{r+n}}\Big(s^p\lt\{f(t)\}(s)\Big)
-\frac{d^{r+n}}{ds^{r+n}}\sum_{j=1}^p s^{p-j}f^{(j-1)}(0)\right).
$$
If $n+r\geqslant p$, the $(n+r)$th derivative of the second sum becomes zero, so there remains
{\arraycolsep0.11em
\begin{eqnarray*}
t^rf^{(s)}(t)&\risingdotseq&(n+1)_r\frac{(-1)^{n+r}}{(n+r)!}\sum_{j=0}^{r+n}\binom{r+n}j(s^p)^{(j)}\lt\{f(t)\}^{(r+n-j)}(s)\\
&=&(n+1)_r\sum_{j=0}^{r+n}(-1)^j\frac{(s^p)^{(j)}}{j!}(-1)^{n+r-j}\frac{\lt\{f(t)\}^{(r+n-j)}(s)}{(r+n-j)!}\\
&=&(n+1)_r\sum_{j=0}^{r+n}(-1)^j\frac{p^{(j)}}{j!} s^{p-j}\f_{r+n-j}(s)\\
&=&(n+1)_r\sum_{j=0}^p(-1)^j\binom pj s^{p-j}\f_{r+n-j}(s)=\nabla_s^p\f_{r+n}(s).
\end{eqnarray*}}Setting $s=1$ gives rise to the equality $t^rf^{(s)}(t)\risingdotseq(n+1)_r\nabla^pf_{n+r}$.

If $n+r\leqslant p-1$, there follows
{\arraycolsep0.09em
\begin{eqnarray*}
t^r f^{(s)}(t)&\risingdotseq&(n+1)_r\sum_{j=0}^{r+n}(-1)^j\binom pj s^{p-j}\f_{r+n-j}(s)
-\frac{(-1)^{n+r}}{n!}\frac{d^{r+n}}{ds^{r+n}}\sum_{j=1}^p s^{p-j}f^{(j-1)}(0)\\
&=&(n+1)_r\sum_{j=0}^{r+n}(-1)^j\binom pj s^{p-j}\f_{r+n-j}(s)\\
&&\hskip3cm-(-1)^{n+r}\sum_{j=1}^{p-r-n}\frac{(p-j)^{(r+n)}}{n!}s^{p-j-r-n}f^{(j-1)}(0)\\
&=&(n+1)_r\sum_{j=0}^{r+n}(-1)^j\binom pj s^{p-j}\f_{r+n-j}(s)\\
&&\hskip2cm-(-1)^{n+r}(n+1)_r\sum_{j=1}^{p-r-n}\frac{(p-j)^{(r+n)}}{(r+n)!}s^{p-j-r-n}f^{(j-1)}(0).%\\
\end{eqnarray*}
}As in the proof of Theorem \ref{th2}, we allow $j$ in the first sum to run up to $r+n$, so that we obtain
$$
t^rf^{(s)}(t)\risingdotseq(n+1)_r\Big(\nabla_s^p\f_{r+n}(s)
-(-1)^{n+r}\sum_{j=1}^{p-r-n}\binom{p-j}{r+n}s^{p-j-r-n}f^{(j-1)}(0)\Big),
$$
which was to be proved. \end{proof}

\section{The Laplace-type transform and fractional\\ derivative}

For $m\leqslant\a<m+1$, the Riemann-Liouville definition the \textbf{fractional derivative} of
a function $f(t)$ is (see \cite{Podlubny})
\begin{equation}\label{diff-alpha}
D^\a f=\frac1{\G(1-\a)}\frac{d^{m+1}}{dt^{m+1}}\int_0^t(t-\t)^{m-\a}f(\t)\,d\t,
\end{equation}
where $D$ is a differential operator. By performing repeatedly integration by parts and differentiation,
this gives
$$
D^\a f=\sum_{k=0}^m\frac{t^{k-\a}f^{(k)}(0)}{\G(k-\a+1)}
+\frac1{\G(m-\a+1)}\int_0^t(t-\t)^{m-\a}f^{(m+1)}(\t)\,d\t.
$$

Let $m=0$ and $g(t)$, denote the integral in \eqref{diff-alpha}, and $G(s)=\lt\{g\}(s)$.
Knowing that $\lt\{g'\}(s)=s\,G(s)-g(0+)$, and because $g(0+)=0$, we find
$$
\lt\{D^\a f\}(s)=\frac{s\,G(s)}{\G(1-\a)}=\frac{s}{\G(1-\a)}\lt\{t^{-\a}*f\}(s)=s\lt\{t^{-\a}\}\lt\{f\}(s)=s^\a F(s).
$$

Replacing $f(t)$ with $D^\a f(t)$ in \eqref{1}, according to \eqref{3}, we obtain %$\lt(D^\a f(t))=s^\a F(s)$,
\begin{equation}\label{4}
\begin{array}{r@{\,=\,}l}
\gt\{D^\a f\}(s)=\dfrac{(-1)^n}{n!}(s^\a F(s))^{(n)}
&\dfrac{(-1)^n}{n!}\dsum_{k=0}^n\binom nk(s^{\a})^{(k)}F^{(n-k)}(s)\\
&\dfrac{(-1)^n}{n!}\dsum_{k=0}^n\binom nk\a^{(k)}s^{\a-k}F^{(n-k)}(s).
\end{array}
\end{equation}
where $\a^{(k)}=\a(\a-1)\cdots(\a-k+1)$ is \textbf{falling factorial}.
After rearranging \eqref{4}, and relying on \eqref{3}, we rewrite it as follows
$$
\gt\{D^\a f\}(s)=\sum_{k=0}^n(-1)^k\frac{\a^{(k)}}{k!}
s^{\a-k}\frac{(-1)^{n-k}}{(n-k)!}F^{(n-k)}(s)=\sum_{k=0}^n(-1)^k\binom{\a}{k}s^{\a-k}\f_{n-k}(s).
$$
So, the Laplace-type transform maps the fractional derivative $D^\a f$ to a sum in terms of the sequence $\{\f_n(s)\}$.

\section{Applications}

We shall demonstrate how by applying the Laplace-type transform one can solve difference equations and obtain combinatorial
identities based on orthogonal polynomials.

\subsection{Solving difference equations}

The Laplace-type transform $\gt$ and its inverse transform $\rt$ we can use to solve some difference equations.

\begin{example}\sl Fibonacci numbers are solutions of the difference equation $F_n=F_{n-1}-F_{n-2}$.
So, we deal first with the equation
$$
\f_n(s)-\f_{n-1}(s)-\f_{n-2}(s)=0,
$$
and by using \eqref{fn-p-nabla-p-fn}, it is transformed into
$$
\f_n(s)-s\f_n(s)+\nabla_s\f_n(s)-s^2\f_n(s)+2s\nabla_s\f_n(s)-\nabla_s^2\f_n(s)=0
$$
or
$$
(1-s-s^2)\f_n(s)+(1+2s)\nabla_s\f_n(s)-\nabla_s^2\f_n(s)=0
$$
where $Q_2(s)=1-s-s^2$. By setting $s=1$, it leads to the difference equation
$$
F_n-3\nabla F_n+\nabla^2F_n=0
$$
that is a result of applying {\rm Theorem \ref{th2}}, i.e. $\gt\{y(t)\}=F_n$, $\gt\{y'\}=\nabla f_n$,
$\gt\{y''\}=\nabla^2 F_n$, to the differential equation
$$
y''-3y'+y=0.
$$

We shall solve Cauchy's problem $y(0)=0\quad y'(0)=1$. At first, we find its general solution
$$
y=C_1e^{\frac{3+\sqrt5}2t}+C_2e^{\frac{3-\sqrt5}2t},
$$
and taking into account Cauchy's conditions, we obtain $C_1=\frac1{\sqrt5}$,
$C_2=-\frac1{\sqrt5}$, and
$$
y=\frac1{\sqrt5}e^{\frac t2(3+\sqrt5)}-\frac1{\sqrt5}e^{\frac t2(3-\sqrt5)}
$$
as a particular solution. However, applying the Laplace-type transform to this particular solution of
the differential equation and using {\rm Lemma \ref{lm1}}, we find the particular solution of the Fibonacci
difference equation
$$
F_n=\frac1{\sqrt5}\Big(\frac{\sqrt5+1}2\Big)^n-\frac1{\sqrt5}\Big(\frac{1-\sqrt5}2\Big)^n,\quad n=0,1,2,\dots,
$$
which is the well-known Binet formula.
\end{example}

\begin{example}\sl Consider the difference equation
\begin{equation}\label{*}
a_0f_n+a_1f_{n-1}+\cdots+a_pf_{n-p}=g_n,\quad a_0,a_1,\dots,a_p\in\mathbb R.
\end{equation}
We regard it as a special case of the functional difference equation
$$
a_0\f_n(s)+a_1\f_{n-1}(s)+\cdots+a_p\f_{n-p}(s)=\gamma_n(s)
$$
for $s=1$. Applying \eqref{fn-p-nabla-p-fn}, we find
$$
Q_p(s)\f_n(s)-\frac{Q'_p(s)}{1!}\nabla_s\f_n(s)+\frac{Q''_p(s)}{2!}\nabla_s^2\f_n(s)
+\cdots+(-1)^p\frac{Q_p^{(s)}(s)}{p!}\nabla_s^p\f_n(s)=\gamma_n(s),
$$
where $Q_p(s)=a_0+a_1s+a_2s^2+\cdots+a_ps^p$. Taking $s=1$ gives rise to the difference equation
\begin{equation}\label{**}
b_0f_n+b_1\nabla f_n+b_2\nabla^2f_n+\cdots+b_{p-1}\nabla^{p-1}+b_p\nabla^pf_n=\gamma_n,
\end{equation}
where $b_k=(-1)^k\dfrac{Q_p^{(k)}(1)}{k!},\,k=0,1,\dots,p$. Thus, \eqref{*} and \eqref{**} are equivalent.
Relying on {\rm Theorem \ref{th2}}, using $\gt\{y^{(k)}(t)\}=\nabla^kf_n$, $k=0,1,\dots,p$,
we conclude that the differential equation
$$
L[y]=b_py^{(s)}+b_{p-1}y^{(p-1)}+\cdots+b_2y''+b_1y'+b_0y=g(t),
$$
gives rise to the difference equation \eqref{**}. We search for a solution of Cauchy's problem $L[Y]=0$,
$Y^{(k)}(0)=0$, $k=0,1,\dots,n-2$, $Y^{(n-1)}(0)=1$, which is in {\rm \cite{Kamke}} given by
$$
y=\int_0^t Y(t-\tau)g(\tau)\,d\tau.
$$
If we now apply {\rm Lemma \ref{lm3}}, so that $\gt\{y(t)\}=f_n$, $\gt\{Y(t)\}=Y_n$ and $\gt\{g(t)\}=g_n$,
we solve the difference equation \eqref{*} in the form of
$$
f_n=\sum_{k=0}^n Y_{n-k}g_k.
$$
\end{example}

\begin{example}\sl Consider the system of difference equations
$$
\nabla_s^2X_n=Y_n,\quad\nabla_s^2Y_n=-2X_n.
$$
Applying the $\rt$ transform to this system, we come to the system of differential equations
$$
\frac{d^2x}{dt^2}=y,\quad\frac{d^2y}{dt^2}=-2x.
$$
The general solution of this system is
$$
x=e^t(C_1\cos t+C_2\sin t),\quad y=e^t(C_1\sin t-C_2\cos t),
$$
where $C_1$ and $C_2$ are arbitrary constants. If the Laplace-type transform is applied to this solution,
the general solution of the above system of difference equations is obtained, i.e.
$$
X_n=(C_1\cos\a+C_2\sin\a)A(s),\quad Y_n=(C_1\sin\a+C_2\cos\a)A(s),
$$
where $A=\dfrac1{(\sqrt{s^2s+2})^{n+1}}$, $\a=(n+1)\arctg\dfrac1{s-1}$.
\end{example}

\begin{example}\sl If the Laplace-type transform is applied to the integral equation {\rm (see \cite{Makarenko})}
$$
\int_0^x(x-t)^\b f(t)dt=x^\lambda\qquad(\lambda\geqslant0,\,\b\geqslant-1,\quad\lambda,\b\in\mathbb R),
$$
it is mapped to the discrete equation
$$
\sum_{k=0}^n\frac{\G(n-k+\b+1)}{s^{\b-k}}\f_k(s)=\frac{\G(n+\lambda+1)}{s^{\lambda}}.
$$
The solution of the integral equation is given by
$$
f(x)=\dfrac{\G(\lambda+1)}{\G(\b+1)\G(\lambda-\b)}\,x^{\lambda-\b-1}\qquad(\lambda-\b+k\ne0,\quad k\in\mathbb N_0),
$$
and applying the Laplace-type transform to it, the solution of the discrete equation is obtained
$$
\f_k(s)=\dfrac{\G(\lambda+1)\G(n+\lambda-\b+2)}{n!\G(\b+1)\G(\lambda-\b)}\,s^{n+\lambda-\b-1}\qquad(\lambda-\b+k\ne0,\quad k\in\mathbb N_0).
$$
\end{example}% Vandermonde's identity
%$\binom{m+n}r=\sum\limits_{k=0}^r\binom mk\binom n{r-k}$, counting bit strings $\binom{n+1}{r+1}=\sum\limits_{k=r}^n\binom kr$

\begin{example}\sl The solution of the system of integral equations {\rm (see \cite{Makarenko})}
\begin{equation}\label{sys-integr}
\begin{cases}
f_1(t)=1-2\dint_0^t e^{2(t-\t)}f_1(\t)d\t+\dint_0^t f_2(\t)d\t,\\
f_2(t)=4t-\dint_0^t f_1(\t)d\t+4\dint_0^t(t-\t)f_2(\t)d\t
\end{cases}
\end{equation}
is $f_1(t)=e^{-t}-te^{-t}$, $f_2(t)=\frac89e^{2t}+\frac13te^{-t}-\frac89e^{-t}$.

Denote $\gt\{f_1\}(s)=\{\f_n(s)\}$ and $\gt\{f_2\}(s)=\{\psi_n(s)\}$. We apply the Laplace-type transform,
and use {\rm Lemma \ref{lm2}} and {\rm Lemma \ref{lm3}} to map \eqref{sys-integr} to the system of equations
\begin{equation}\label{sys-fi-psi}
\begin{cases}
\f_n(s)=\dfrac1{s^{n+1}}-2\dsum_{k=0}^n\frac{\f_k(s)}{(s-2)^{n-k+1}}+\dsum_{k=0}^n\frac{\psi_k(s)}{s^{n-k+1}},\\
\psi_n(s)=\dfrac{4(n+1)}{s^{n+2}}-\dsum_{k=0}^n\frac{\f_k(s)}{s^{n-k+1}}+4\dsum_{k=0}^n\frac{(n-k+1)\psi_k(s)}{s^{n-k+2}}.
\end{cases}
\end{equation}
By applying the Laplace-type transform to $f_1(t)$ and $f_2(t)$, we find the solution of \eqref{sys-fi-psi}
$$
\begin{cases}
\f_n(s)=\dfrac1{(s+1)^{n+1}}-\frac{n+1}{(s+1)^{n+2}},\\%\quad(n\geqslant2).
\psi_n(s)=\dfrac8{9(s-2)^{n+1}}+\frac{n+1}{3(s+1)^{n+2}}-\dfrac8{9(s+1)^{n+1}}.
\end{cases}
$$
\end{example}

\begin{example}Consider the integro-differential equation {\rm (see \cite{Makarenko})}
$$
f''(t)+\int_0^t e^{2(t-\t)}f'(\t)d\t=e^{2t},\quad f(0)=0,\,f'(0)=0.
$$
We rely on {\rm Theorem \ref{th2}} for $n\geqslant p$, where $p=2$, and apply the Laplace-type transform
to map it to the equation
\begin{equation}\label{integro-diff}
\nabla_s^2\f_n(s)+\sum_{k=0}^n\frac{\nabla_s\f_k(s)}{(s-2)^{n-k+1}}=\frac1{(s-2)^{n+1}}.
\end{equation}
Applying the Laplace-type transform now to the solution $f(t)=te^t-e^t+1$ of
the integrodifferential equation, we find
$$
\f_n(s)=\frac{n+1}{(s-1)^{n+2}}-\frac1{(s-1)^{n+1}}+\frac1{s^{n+1}}\quad(n\geqslant2),
$$
which is the solution of \eqref{integro-diff}.
\end{example}

\subsection{Some identities based on orthogonal polynomials}

Relying on the classical orthogonal polynomials and applying the Laplace-type transform
leads us to some interesting combinatorial identities.

\begin{example}\sl The explicit representation of the standard Laguerre polynomials is
$$
L_m(x)=\sum_{k=0}^m\frac{(-1)^k}{k!}\binom mk x^k.
$$
Applying the Laplace-type transform, they are mapped to
$$
L_m(x)\risingdotseq\sum_{k=0}^m\frac{(-1)^{n+k}}{s^{n+k+1}}\binom mk\binom{n+k}{k}. %\quad (n\geqslant m).
$$
Taking into account that and applying the Laplace-type transform to the algebraic equality involving
Laguerre polynomials {\rm (see \cite{Suetin})}
$$
\sum_{k=0}^m(-1)^k\binom mk L_k(x)=\frac{x^m}{m!},
$$
the following combinatorial identity is obtained
$$
\sum_{k=0}^m\binom mk\sum_{j=0}^k\frac{(-1)^{k+j}}{s^j}\binom kj\binom{n+j}j=\binom{m+n}n\frac1{s^m}.
$$
\end{example}

\begin{example}\sl To apply the Laplace-type transform to Legendre polynomials
$$
P_n(x)=\frac1{2^n}\sum_{k=0}^{[\frac n2]}(-1)^k\binom nk\binom{2n-2k}n x^{n-2k},\quad x\in(-1,1),
$$
we first have to map the interval $(-1,1)$ on to the interval $(0,+\infty)$. For that purpose, we make
use of the function $x=1-2e^{-t}$, whereby
$$
p_n(t)=P_n(1-2e^{-t})=\frac1{2^n}\sum_{k=0}^{[\frac n2]}(-1)^k\binom nk\binom{2n-2k}n
\sum_{j=0}^{n-2k}\binom{n-2k}j(-2e^{-t})^j,
$$
which, after interchanging the order of summation, can be rewritten in the form of
$$
p_n(t)=\sum_{j=0}^n e^{-jt}
\sum_{k=0}^{[\frac n2]}\frac{(-1)^{j+k}}{2^{n-j}}\binom nk\binom{2n-2k}n\binom{n-2k}j
=\sum_{j=0}^n A_{n,j}\,e^{-jt},
$$
where
\begin{equation}\label{1*}
A_{n,j}=\sum_{k=0}^{[\frac n2]}\frac{(-1)^{j+k}}{2^{n-j}}\binom nk\binom{2n-2k}n\binom{n-2k}j.
\end{equation}

By means of the function $x=1-2e^{-t}$ Bonnet's recursion formula of Legendre polynomials
$$
(m+1)P_{m+1}(x)=(2m+1)xP_m(x)-m P_{m-1}(x)
$$
is transformed to
$$
(m+1)p_{m+1}(t)=(2m+1)(1-2e^{-t})p_m(t)-m p_{m-1}(t),
$$
so, we have
$$
(m+1)\sum_{j=0}^{m+1}A_{m+1,j}e^{-jt}=(2m+1)(1-2e^{-t})\sum_{j=0}^mA_{m,j}e^{-jt}
-m\sum_{j=0}^{m-1}A_{m-1,j}e^{-jt}.
$$
After rearrangement and comparing corresponding coefficients, we obtain combinatorial identities. First of all, simple ones. The equalities $(m+1)A_{m+1,m+1}=-2(2m+1)A_{m,m}$ and $(m+1)A_{m+1,m}=(2m+1)(A_{m,m}-2A_{m,m-1})$ both yield
\medskip

$\displaystyle 1^\circ\quad\binom{2m+2}{m+1}=\frac{4m+2}{m+1}\binom{2m}{m}$,
\medskip

\noindent but the equality $(m+1)A_{m+1,0}=(2m+1)A_{m,0}-mA_{m-1,0}$ yields
\begin{multline*}
\hskip0.2cm2^\circ\quad\sum_{k=0}^{[\frac{m+1}2]}(-1)^k\binom{m+1}k\binom{2m+2-2k}{m+1}
=\frac{4m+2}{m+1}\dsum_{k=0}^{[\frac m2]}(-1)^k\binom mk\binom{2m-2k}m\\
-\frac{4m}{m+1}\sum_{k=0}^{[\frac{m-1}2]}(-1)^k\binom{m-1}k\binom{2m-2-2k}{m-1}.
\end{multline*}
More generally, the equality $(m+1)A_{m+1,j}=(2m+1)(A_{m,j}-2A_{m,j-1})-mA_{m-1,j}$, where $j=1,2,\ldots,m-1$, yields
\begin{multline*}
\hskip0.2cm3^\circ\quad\sum_{k=0}^{[\frac{m+1}2]}(-1)^{k}\binom{m+1}k\binom{2m+2-2k}{m+1}\binom{m+1-2k}j\\
=\frac{4m+2}{m+1}\sum_{k=0}^{[\frac m2]}(-1)^{k}\binom mk\binom{2m-2k}m\binom{m-2k+1}j\\
-\frac{4m}{m+1}\sum_{k=0}^{[\frac{m-1}2]}(-1)^{k}\binom{m-1}k\binom{2m-2-2k}{m-1}\binom{m-1-2k}j.
\end{multline*}
\end{example}

\section{Appendix}

\begin{center} {\tabcolsep 6pt \renewcommand{\arraystretch}{1.95}
\begin{tabular}{|c|c|}
\hline
$f(t)$ & $f(t)\risingdotseq\f_n(s)=\dfrac1{\G(n+1)}\int_0^\infty e^{-st}t^nf(t)\,dt$\\
\hline
$e^{at}$ & $\dfrac1{(s-a)^{n+1}}$ \\
\hline
$e^{at}f(t)$ & $\f_n(s-a)$\\
\hline
$t^a,\quad(a>-1)$ & $\dfrac{\G(a+n+1)}{s^{n+a+1}\G(n+1)}$\\
\hline
$t^ae^{-bt},\quad(a>-1)$ & $\dfrac{\G(a+n+1)}{(s+b)^{n+a+1}\G(n+1)}$\\
\hline
$\sin at$ & $\dfrac1{\sqrt{(s^2+a^2)^{n+1}}}\sin\left((n+1)\arctg\frac as\right)$\\
\hline
$\cos at$ & $\dfrac1{\sqrt{(s^2+a^2)^{n+1}}}\cos\left((n+1)\arctg\frac as\right)$\\
\hline
$f(t-a),\quad(a>0)$ & $e^{-as}\dsum_{k=0}^n\frac{a^{n-k}}{(n-k)!}\f_{n-k}(s)$\\
\hline
$\ln t$ & $\dfrac1{s^{n+1}}(H_n-\gamma-\ln s),\quad\gamma=\lim\limits_{n\to\infty}(H_n-\ln n),
\quad H_n=\dsum_{k=1}^n\frac1k$\\
\hline
\end{tabular}}
\end{center}

%6. Conclusion: It is essential to provide a concluding section summarizing the findings and implications of the work.


\begin{thebibliography}{14}
\bibitem{Kamke} E. Kamke, {\it Differentialgleichungen. L\"osungsmethoden und L\"osungen}, Leipzig, 1959.

\bibitem{Makarenko} M. Krasnov, A. Kiselev, G. Makarenko, {\it Problems and Exercises in integral equations},
Mir Publishers, Moscow, 1971.

\bibitem{Mitr} D. Mitrinovi\'c, J. Ke\v cki\'c, {\it The Cauchy Method of Residues - Theory and Applications},
D. Reidel Publishing Company, Dordrecht, Holland, 1984.

\bibitem{Suetin} P. K. Suetin, {\it Classical Orthogonal Polynomials}, Nauka, Moscow, 1979.

%\bibitem{Magnus} A. Erd\'elyi, W. Magnus, F. Oberhettinger, F. Tricomi, {\it Tables of Integral Transforms},
%vol. 1, McGraw-Hill, New York, 1954.

\bibitem{Podlubny} I. Podlubny, {\it Fractional Differential Equations}, Academic Press, 1999.

\bibitem{Shiff} J. Shiff, {\it The Laplace Transform: Theory and Applications}, Springer-Verlag,
New York, 1999.

\bibitem{Temme} N. M. Temme, {\it Uniform Asymptotic Methods for Integrals}, Indagationes Mathematicae Vol. 24 No. 4 (2013) 739-765.
\end{thebibliography}
\end{document}